\documentclass[a4paper,english,cleveref,autoref,final]{lipics-v2019}

\newif\ifproc
\procfalse

\nolinenumbers

\usepackage{xspace}
\usepackage{bm}
\usepackage{hyperref}
\usepackage{csquotes}
\usepackage{tikz}
\usepackage{microtype}
\usepackage{nicefrac}
\usepackage{mathtools}
\usepackage{etoolbox}
\BeforeBeginEnvironment{enumerate*}{%
	\setlist[enumerate,1]{label=(\alph*), font={\bfseries}}
}
\AfterEndEnvironment{enumerate*}{%
	\setlist[enumerate,1]{label=\arabic*.}
}

\usepackage[sort&compress,numbers]{natbib}

\usepackage{tikz}
\usetikzlibrary{fit,patterns,decorations.pathreplacing}

\title{A Conjecture on Rainbow Hamiltonian Cycle Decomposition} 

\titlerunning{}

\author{Ramin Javadi}{Isfahan University of Technology, Isfahan, Iran}{rjavadi@iut.ac.ir}{https://orcid.org/0000-0003-4401-2110}{}

\author{Meysam Miralaei}{Institute for Research in Fundamental Sciences (IPM), Tehran, Iran}{m.miralaei@ipm.ir}{}{This research was supported by a grant from IPM.}

\authorrunning{R.\ Javadi and M. Miralaei}

\Copyright{Ramin Javadi and Meysam Miralaei}


\category{}

\supplement{}

\ifproc\else
\hideLIPIcs  
\fi

\newcommand{\hcd}{Hamiltonian cycle decomposition}

\DeclareMathOperator{\la}{la}

\begin{document}

\maketitle
\begin{abstract}
Wu in 1999 conjectured that if $H$ is a subgraph of the complete graph $K_{2n+1}$ with $n$ edges, then there is a Hamiltonian cycle decomposition of $K_{2n+1}$ such that each edge of $H$ is in a separate Hamiltonian cycle. The conjecture was partially settled by Liu and Chen (2023) in cases that $|V(H)|\leqslant n+1$, $H$ is a linear forest, or $n\leqslant 5$. In this paper, we settle the conjecture completely. This result can be viewed as a complete graph analogous of Evans conjecture and  has some applications in linear arboricity conjecture and restricted size Ramsey numbers. 
\end{abstract}



\vskip 1cm 

\section{Introduction}
An (edge) decomposition of a graph $G=(V,E)$ is a partition $\mathcal{C}=(C_1,\ldots,C_n)$ of the edge set $E$ into disjoint subsets $C_i$'s. By $|C_i|$ we mean the number of edges in $C_i$. Given a decomposition $\mathcal{C}=(C_1,\ldots,C_n)$  and a subgraph $H$ of $G$, we say that $H$ is rainbow in $\mathcal{C}$ if no two edges of $H$ are contained in the same set $C_i$, for some $i$.
A Hamiltonian cycle (resp. path) decomposition of a graph $G=(V,E)$ is a decomposition of $G$ into subsets $C_1,\ldots,C_n \subseteq E$ such that each $C_i$ induces a Hamiltonian cycle (resp. path) in $G$.

It is well-known that the complete graph $K_{2n+1}$ has a decomposition into $n$ Hamiltonian cycles. In 1982, Hilton \cite{hilton}, in a seminal work, proposed a procedure for generating all Hamiltonian cycle decompositions of $K_{2n+1}$ and found necessary and sufficient conditions for extending a decomposition of $K_r$ (for some $r\leqslant 2n$) to a \hcd\ of $K_{2n+1}$. 

Wu \cite{wu}, while studying linear arboricity conjecture, raised the question that for any subgraph $H$ of $K_{2n+1}$ with $n$ edges, if the complete graph $K_{2n+1}$ has a \hcd\ in which $H$ is rainbow? 

\begin{conj} {\rm \cite{wu}} \label{conj:main}
Let $H$ be a subgraph of $K_{2n+1}$ with $n$ edges. Then, $K_{2n+1}$ admits a Hamiltonian cycle decomposition $\mathcal{C}$ such that $H$ is rainbow in $\mathcal{C}$.
\end{conj}

In \cite{liu-chen}, Liu and Chen partially proved  Conjecture~\ref{conj:main} when $n\leqslant 5$, or $H$  has at most $n+1$ vertices, or $H$ is a \textit{linear forest}, where a {linear forest} is a disjoint union of paths.

\begin{theorem} {\rm \cite{liu-chen}} \label{thm:liu}
	Suppose that $H$ is  a subgraph of $K_{2n+1}$ with $n$ edges such that either
	\begin{itemize}
		\item $H$ has at most $n+1$ vertices, or
		\item $H$ is a linear forest, or 
		\item $n\leqslant 5$.
	\end{itemize}
	Then, there is a Hamiltonian cycle decomposition $\mathcal{C}$ of $K_{2n+1}$ such that $H$  is rainbow in $\mathcal{C}$.
\end{theorem}

There are also some partial results regarding Conjecture~\ref{conj:main} in \cite{teerasak}. In this paper, we settle Conjecture~\ref{conj:main} completely. The conjecture has some applications in linear arboricity conjecture as well as Ramsey theory.

 The \textit{linear arboricity} of a graph $G$ denoted by $\la(G) $  is the minimum number $k$ such that the edge-set of $G$ can be decomposed into $k$ linear forest. Linear arboricity conjecture \cite{harrary} asserts that if $G$ has maximum degree $\Delta$, then $\la(G)\leqslant \lceil (\Delta+1)/2\rceil$ (see e.g. \cite{fox,lang} for the literature). A graph $G$ is called \textit{critical} if $\la(G)\geqslant  \lceil (\Delta(G)+1)/2\rceil$ and for any edge $e\in E(G)$, $\la(G)> \la(G\setminus e)$. One can check that if Conjecture~\ref{conj:main} is true, then removing any $n-1$ edges from $K_{2n+1}$ yields a critical graph. 
It can also be seen that if the conjecture is true, then for every graph $G$ with $n$ vertices and at most $(n-1)^2/2$ edges, $\la(G)\leqslant (n-1)/2$ (see \cite{teerasak}). 

Conjecture~\ref{conj:main} can also be viewed as a complete graph similar result of Evans conjecture \cite{evans}. In fact, Evans conjecture (proved by Smetianuk \cite{smetianuk}) equivalently asserts that 
a proper edge coloring of any subgraph $H$ of $K_{n,n}$ with $n$ edges, can be extended to a proper edge coloring of $K_{n,n}$. Andersen and Hilton \cite{anderson} proved analogous results regarding proper edge coloring of complete graphs. In particular, they proved that
\begin{theorem}{\rm \cite{anderson}}\label{thm:anderson}
\begin{itemize}
	\item Let $H$ be a subgraph of $K_{2n}$ with $n-1$ edges. Any proper edge coloring of $H$ can be extended to a proper edge coloring of $K_{2n}$ with $2n-1$ colors.
	\item 
	Let $H$ be a subgraph of $K_{2n+1}$ with $n+1$ edges. Any proper edge coloring of $H$ can be extended to a proper edge coloring of $K_{2n+1}$ with $2n+1$ colors.
\end{itemize}
 \end{theorem}

 One can view Conjecture~1 as an analogous result of Theorem~\ref{thm:anderson} where proper edge coloring is replaced with \hcd\ and $H$ is colored with $n$ different colors. Therefore, finding necessary and sufficient conditions for extending an arbitrary coloring of $H$ to a \hcd\ of $K_{2n+1}$ is an interesting problem and can be viewed as a generalization of Conjecture~\ref{conj:main} (see Concluding Remarks and Conjecture~\ref{conj:evans}).   

Another application of Conjecture~\ref{conj:main} is in Ramsey theory. Given graphs $G_1,\ldots, G_k$, the \textit{Ramsey number} $R(G_1,\ldots, G_k)$ is the minimum number $n$ such that in every edge coloring of $K_n$ with $k$ colors, there is a copy of $G_i$ with color $i$, for some $i\in [k]$. The \textit{restricted size Ramsey number} denoted by $\hat{r}^*(G_1,\ldots,G_k)$ is the minimum number $m$ such that there is a graph $G$ with $R(G_1,\ldots,G_k)$ vertices and $m$ edges, such that in every edge coloring of $G$ with $k$ colors, there is a copy of $G_i$ with color $i$ for some $i\in [k]$. 
Miralaei and Shahsiah in \cite{miralaei} proved some upper bounds for the restricted size Ramsey numbers of stars versus cliques and conjectured that equality holds for these upper bounds. They mentioned that if Conjecture~\ref{conj:main} is correct, then their conjecture is proved for the case that the number of stars of size even is positive and even.
So, combining our result with their proof can yield the following result. 

\begin{theorem}
Let $m,n, p_1,\ldots, p_m, q_1,\ldots, q_n $ be positive integers and $s$ be the number of even integers in the set $\{q_i\}_{i\in [n]}$. If $s$ is positive and even, then
\[
\hat{r}^*(K_{1,q_1},\ldots, K_{1,q_n},K_{p_1},\ldots, K_{p_m})= \binom{t}{2}-\dfrac{s}{2}+1,  
\]  
where $t= R(K_{1,q_1},\ldots, K_{1,q_n},K_{p_1},\ldots, K_{p_m})$.
\end{theorem}

In order to prove Conjecture~\ref{conj:main}, we need the following result from \cite{hilton} which deals with the extension of an edge decomposition of $K_r$ to a \hcd\ of $K_{2n+1}$. 

\begin{theorem}{\rm \cite{hilton}}\label{thm:hilton}
Let $1\leqslant r\leqslant 2n$ be an integer. A decomposition $\mathcal{C}=(C_1,\ldots,C_n)$ of $K_r$ can be extended to a \hcd\ $\mathcal{C}'=(C'_1,\ldots,C'_n)$ of $K_{2n+1}$ where $C_i\subseteq C'_i$, for each $i\in[n]$, if and only if the induced subgraph of $G$ on each $C_i$ is a linear forest with at most $2n+1-r$ disjoint paths $($counting a vertex of $K_r$ with no edge in $C_i$ as a path of lengths $0)$.
\end{theorem}

Note that any linear forest on $r$ vertices with $e$ edges has exactly $r-e$ connected components. Therefore, Theorem~\ref{thm:hilton} can be reformulated as follows.

\begin{corollary}{\rm \cite{hilton}}\label{cor:hilton}
	Let $1\leqslant r\leqslant 2n$ be an integer. An edge decomposition $\mathcal{C}=(C_1,\ldots,C_n)$ of $K_r$ can be extended to a \hcd\ $\mathcal{C}'=(C'_1,\ldots,C'_n)$ of $K_{2n+1}$ where $C_i\subseteq C'_i$, for each $i\in[n]$, if and only if
	\begin{enumerate}
		\item for each $i\in [n]$, the induced subgraph of $K_r$ on $C_i$ is a linear forest, and
		\item for each $i\in [n]$, $|C_i|\geqslant 2r-2n-1$.
	\end{enumerate}
\end{corollary}

To prove Conjecture~\ref{conj:main}, we take a minimal counterexample $H$ to the conjecture with $n$ edges, in the sense that $n$ is the smallest integer for which the conjecture does not hold and among all counterexamples with $n$ edges, $H$ is the graph with the smallest number of vertices. 
We give rise to a contradiction in two phases. First, in  Section~\ref{sec:dense}, we embed the components of $H$ which are non-isomorphic to $K_2$ into an edge decomposition of $K_r$ for some $r$ with desired properties and then in Section~\ref{sec:sparse}, we extend this decomposition to a \hcd\ of $K_{2n+1}$ in which $H$ is rainbow. This comes to a contradiction and proves Conjecture~\ref{conj:main}.

\subsection{Notations and Terminologies} 
For positive integers $m,n$, $m\leqslant n$, the sets $\{1,\ldots, n\} $ and $\{m,\ldots, n\} $ are respectively denoted by $[n]$ and $[m,n]$. Given a graph $G=(V(G),E(G))$, $v(G)$ and $e(G)$ respectively stand for $|V(G)|$ and $|E(G)|$. For a subset $F\subseteq E(G)$, the induced subgraph of $G$ on $F$ is denoted by $G[F]$. For any vertex $v\in V(G)$, the set of neighbors of $v$ in $G$ and $G[F]$ are respectively denoted by $N_G(v)$ and $N_F(v)$. Also, $\deg_G(v)$ and $\deg_F(v)$ stand for the number of edges incident with $v$ in $E(G)$ and $F$, respectively. For two disjoint subsets $A,B\subset V(G)$, $e_G(A,B)$ stands for the number of edges in $G$ with one endpoint in $A$ and one endpoint in $B$ and we define $e_F(A,B)=e_{G[F]}(A,B)$. Also, we define $N_F(A)=\cup_{a\in A} N_F(a)$. If $H$ is a subgraph of $G$, then $G\setminus H$ is the graph obtained from $G$ by removing all edges of $H$.  

\section{Embedding Components Non-isomorphic to $K_2$} \label{sec:dense}
Let $H$ be a minimal counterexample to Conjecture~\ref{conj:main} with $e(H)=n$. Also, let $H'$ be the subgraph of $H$ consisting of all connected components of $H$ which are non-isomorphic to $K_2$. In this section, we prove that there exists an edge decomposition of the complete graph on $v(H')$ vertices with some desired properties in which $H'$ is rainbow. Then, in the next section, we extend such decomposition to a \hcd\ of $K_{2n+1}$ in which $H$ is rainbow.  
%

\begin{theorem} \label{thm:dense}
Let $H$ be a minimal counterexample to Conjecture~\ref{conj:main} with $n$ edges. Also let $H'$ be the subgraph of $H$ containing all its connected components which are not isomorphic to $K_2$ and suppose that $H'$ has $r$ vertices and $t$ edges. Then, the complete graph $K_r$ admits an edge decomposition $(P_1,\ldots, P_n)$ such that
\begin{itemize}
    \item[(1)] For each $i\in [n]$, $P_i$ induces a linear forest in $K_r$.
    \item[(2)] $H'$ is rainbow in $(P_1,\ldots,P_t)$.
    \item[(3)] For each $i\in [t]$, $|P_i|\geqslant 2r-2n-1$.
    \item[(4)] For each $i\in [t+1,n]$, $|P_i|\geqslant 2r-2n$.
\end{itemize}
\end{theorem}

\begin{proof}
Due to Theorem~\ref{thm:liu}, $H$ is not a linear forest and $n\geqslant 6$. Also, every connected component of $H'$ has at least two edges. Thus, $H'$ has at most $(t-1)/2$ connected components  and so $r\leqslant t+(t-1)/2=(3t-1)/2$. Also, we can assume that $r \geqslant n+1$. To see this, suppose that $r\leqslant n$. Then, we can extend $H'$ to a graph $\widehat{H}$ such that $v(\widehat{H})\leqslant e(\widehat{H})=n$. Thus, by Theorem~\ref{thm:liu}, $K_{2n+1}$ admits a \hcd\  $\mathcal{C}=(C_1,\ldots, C_n)$ such that $\widehat{H}$ is rainbow in $\mathcal{C}$. Without loss of generality, we can assume that $H'$ is rainbow in $C_1,\ldots,C_t$. Also, since $v(H')=r$, we can remove $2n+1-r$ vertices from $K_{2n+1}$ to obtain an edge decomposition $(P_1,\ldots, P_n)$ of $K_r$ where $H'$ is rainbow in $P_1,\ldots, P_t$. Conditions (3) and (4) trivially hold as $r\leqslant n$. Therefore, we can assume that $r\geqslant n+1$. 

Now, let $s= \lceil r/2\rceil $ and $H''$ be a subgraph of $H'$ with $e(H'')= s$. Since $H$ is a minimal counterexample and $H''$ has fewer edges than $H$, there is a 
\hcd\ $\mathcal{C}=(C_1,\ldots,C_{ s })$ for $K_{2s+1}$ such that $H''$ is rainbow in $\mathcal{C}$. Since $v(H'')\leqslant v(H')=r$, we can remove from $K_{2s+1}$ one vertex whenever $r$ is even and two vertices whenever $r$ is odd to obtain an edge decomposition $\mathcal{P}=(P_1,\ldots, P_{s})$ of $K_{r}$ in which $H''$ is rainbow and each $P_i$ contains $r-1$ edges when $r$ is even and at least $r-2$ edges when $r$ is odd. Since $v(H')=r$, we can assume that $H'$ is a subgraph of $K_{r}$. The subgraph $H'\setminus H''$ has exactly $t-s$ edges which are distributed in the subsets $P_i$'s. Now, define $t-s$ new empty subsets $P_{s+1}, \ldots, P_{t}$ and move each edge of $H'\setminus H''$ from $P_1,\ldots,P_{s}$ to a separate set $P_i$, $s+1\leqslant i\leqslant t$. Also, define $P_{t+1},\ldots,P_n$ as empty subsets. By abuse of notation, let $\mathcal{P}=(P_1,\ldots, P_n)$ be the obtained decomposition of $K_{r}$ where $H'$ is rainbow in $P_1,\ldots,P_{t}$. Also, without loss of generality, assume that $|P_1|\geqslant |P_2|\geqslant \cdots\geqslant |P_{s}|\geqslant |P_{s+1}|=\cdots=|P_{t}|=1> |P_{t+1}|=\cdots=|P_n|=0$. We are going to move some edges from large $P_i$'s to small $P_i$'s to ensure that each $P_i$ induces a linear forest such that 
\begin{align}
|P_i|\geqslant \begin{cases}
2r-2n-1 & 1\leqslant i\leqslant t, \\ 
2r-2n & t+1\leqslant i\leqslant n.
\end{cases}
\label{eq:pi}
\end{align}

Now, we consider the following two cases.

\paragraph*{Case 1. $n+1\leqslant r\leqslant (4n-1)/3$.} 
First, we claim that $|P_{n-s}|\geqslant 4r-4n-1$. To prove the claim, by the contrary, suppose that $|P_{n-s}|\leqslant 4r-4n-2$. Therefore, $|P_i|\leqslant 4r-4n-2$, for all $n-s\leqslant i\leqslant s$. Since before moving edges of $H'\setminus H''$, each $P_i$ has at least $r-2$ edges, we have
\begin{align*}
2n-r\geqslant {2t-r}\geqslant 2\, e(H'\setminus H'')&\geqslant 2\sum_{i=n-s}^{s} \left(r-2-|P_i|\right) \\
&\geqslant 2\sum_{i=n-s}^{s} \left(r-2-4r+4n+2\right)\\
&\geqslant  2(r-n+1)(4n-3r).\\
\end{align*}
Thus, $f_n(r)= 2(r-n+1)(4n-3r)-2n+r\leqslant 0$. On the other hand, when $n+1\leqslant r\leqslant (4n-1)/3$, the function $f_n(r)$ attains its minimum on $r=n+1$ or $r=(4n-1)/3$. However, $f_n(n+1)=3n-11 >0$ as $n\geqslant 6$ and $f_n((4n-1)/3)=1 >0$. This is a contradiction which proves the claim. Therefore, $|P_1|\geqslant \cdots\geqslant |P_{n-s}|\geqslant 4r-4n-1$.

Now, for each $i$, $s+1\leqslant i\leqslant t$, do the following. Let $ab$ be the edge of $H'\setminus H''$ in $P_i$. Consider the set $P_j=P_{i-s}$ and let $e'$ be the unique edge of $H''$ in $P_j$. Also, let $e_1$ and $e_2$ be two edges in $P_j$ incident with $a$ and $b$,  respectively (if exist) such that in $G[P_j\setminus \{e_1,e_2\}]$, $a$ and $b$ are in different connected components. Now, move $2r-2n-2$ edges from $P_j\setminus \{e',e_1,e_2\}$ to $P_i$. Since $|P_j|\geqslant 4r-4n-1\geqslant 2r-2n-2+3$, this can be done. Since $P_i$ has already contained the edge $ab$ and we excluded the edges $e_1$ and $e_2$, so after moving these $2r-2n-2$ edges, $G[P_i]$ is a linear forest.  Also, the remaining edges in $P_j$ are at least $4r-4n-1-(2r-2n-2)=2r-2n+1$. 

Now, for each $i$, $t+1\leqslant i\leqslant n$, do the following. Consider the set $P_j=P_{i-s}$ and let $e'$ be the unique edge of $H''$ in $P_j$. Now, move $2r-2n$ arbitrary edges from $P_j\setminus \{e'\}$ to $P_i$. Since $|P_j|\geqslant 4r-4n-1\geqslant 2r-2n+1$, this can be done. Also, the remaining edges in $P_j$ are at least $4r-4n-1-(2r-2n)=2r-2n-1$. 
On the other hand, for each $i$, $n-s+1\leqslant i\leqslant s$, $|P_i|\geqslant (r-2)-e(H'\setminus H'')=r-2-t+s\geqslant 2r-2n-1$, as $t\leqslant n$.

Hence, every $P_i$ satisfies Condition~\eqref{eq:pi},  $G[P_i]$ is a linear forest and we are done.

\paragraph*{Case 2. $4n/3\leqslant r\leqslant (3t-1)/2$.}
It is clear that in this case $(8n+3)/9\leqslant t\leqslant n$.  
First, we claim that $|P_{2n-2s}|\geqslant 3r-3n-\epsilon$, where $\epsilon=\max\{t-n+1,0\}$. To prove the claim, by the contrary, suppose that $|P_{2n-2s}|\leqslant 3r-3n-\epsilon-1$. Therefore, $|P_i|\leqslant 3r-3n-\epsilon-1$, for all $2n-2s\leqslant i\leqslant s$. Now, define $\delta=0$ when $r$ is even and $\delta=1$ when $r$ is odd. Since before moving edges of $H'\setminus H''$, each $P_i$ has at least $r-1-\delta$ edges, so
\begin{align*}
2t-r-\delta=2t-2s=  2\, e(H'\setminus H'')&\geqslant 2\sum_{i=2n-2s}^{s} \left(r-1-\delta-|P_i|\right) \\
&\geqslant 2\sum_{i=2n-2s}^{s} (r-1-\delta-(3r-3n-\epsilon-1))\\
& \geqslant  (3r-4n+2)(3n-2r-\delta+\epsilon). 
\end{align*}
Thus, $f_{n,t}(r)= (3r-4n+2)(3n-2r-\delta+\epsilon)-2t+r+\delta\leqslant 0$. On the other hand, when $4n/3\leqslant r\leqslant (3t-1)/2$, the function $f_{n,t}(r)$ attains its minimum on $r=4n/3$ or $r=(3t-1)/2$. However, $f_{n,t}(4n/3)=2n-2t-\delta+2\epsilon >0$, as  $\epsilon=1$ whenever $t=n$ and $\epsilon=0$ whenever $t<n$. Also, $2f_{n,t}((3t-1)/2)=(9t-8n+1)(3n-3t+1-\delta+\epsilon)-t+2\delta-1$. Note that if $t=n$, then $\epsilon=1$ and $2f_{n,t}((3t-1)/2)=(n+1)(1-\delta)+2\delta>0$. So, $(8n+3)/9\leqslant t\leqslant n-1$, then $n\geqslant 12$, $\epsilon=0$ and 
the function $g_n(t)=2f_{n,t}((3t-1)/2)$ attains its minimum on $t=(8n+3)/9$ or $t=n-1$. However, $g_n((8n+3)/9)= 4n/9-2\delta-4/3>0$ as $n\geqslant 12$ and $g_n(n-1)= (n-8)(4-\delta)-n+2\delta= n(3-\delta)+10\delta-32>0$ as $n\geqslant 12$.  
This leads to a contradiction and proves the claim.
Therefore, $|P_1|\geqslant \cdots\geqslant |P_{2n-2s}|\geqslant 3r-3n-\epsilon$. 
Now, for each $i$, $s+1\leqslant i\leqslant t$, do the following process. 

Suppose that $P_i=\{ab\}$, where $ab$ is an edge of $H'\setminus H''$. We are going to move $2r-2n-2$ edges from some sets $P_j$ and $P_\ell$ to $P_i$ to guarantee Condition \eqref{eq:pi}. First, consider $P_j=P_{i-s}$ which has at least $3r-3n-\epsilon$ edges. Also, let $e'$ be the unique edge of $H''$ in $P_j$. Also, we can choose at most two edges $e_1,e_2$ in $P_j$ incident with $a,b$, respectively, such that in $G[P_j\setminus \{e_1,e_2\}]$, vertices $a$ and $b$ are in different connected components.
Now, let $P_j'$ be a set of $r-n-2$ arbitrary edges in  $P_j\setminus \{e',e_1,e_2\}$ (it exists since $3r-3n-\epsilon-3\geqslant r-n-2$). Move all edges in $P_j'$ from $P_j$ to $P_i$. Then, $P_i$ now induces a linear forest in $G$ with $r-n-1$ edges. Let $I_i$ and $J_i$ be respectively the set of all inner vertices and endpoints of the paths in $G[P_i]$.

Now, consider $P_\ell= P_{i-2s+t}$ and let $e'$ be the unique edge of $H''$ in $P_\ell$. Now, let $P_\ell'$ be a subset of $P_\ell$ containing all edges incident with $I_i$ along with at most one edge incident with each vertex in $J_i$ such that if $z,w$ are two endpoints of a path in $G[P_i]$ then $z$ and $w$ are in different connected components of $G[P_\ell \setminus P_\ell']$. Since $P_\ell$ induces a linear forest, we have $|P_\ell'|\leqslant 2|I_i|+|J_i|=2|P_i|=2(r-n-1)$. Therefore, since $|P_\ell|\geqslant 3r-3n-1$, we have $|P_\ell\setminus (P'_\ell\cup \{e'\})|\geqslant 3r-3n-1-(2r-2n-2)-1=r-n $. Finally, move $r-n$ arbitrary edges from $P_\ell\setminus (P'_\ell\cup \{e'\})$  to $P_i$. Choosing such edges guarantees that $P_i$ induces a linear forest with $2r-2n-1$ edges. 

Finally, for each $i$, $t+1\leqslant i\leqslant n$, do the following process. First, note that since $t\leqslant n-1$, we have $\epsilon=0$. 
First, consider $P_j=P_{i+t-2s}$. Also, let $e'$ be the unique edge of $H''$ in $P_j$. We choose an arbitrary subset $P'_j\subseteq P_j\setminus\{e'\}$ such that $|P'_j|=r-n-1$. This can be done since $r-n-1\leqslant 3r-3n-1$. Now, move all $r-n-1$ edges in $ P'_j$ to $P_i$. 
Then, $P_i$ is now a linear forest in $G$ with $r-n-1$ edges. Let $I_i$ and $J_i$ be respectively the set of all inner vertices and endpoints of the paths in $G[P_i]$.

Now, consider $P_\ell= P_{i+n-2s}$ and let $e'$ be the unique edge of $H''$ in $P_\ell$. 
Now, let $P_\ell'$ be a subset of $P_\ell$ containing all edges incident with $I_i$ along with at most one edge incident with each vertex in $J_i$ such that if $z,w$ are two endpoints of a path in $G[P_i]$ then $z$ and $w$ are in different connected components of $G[P_\ell \setminus P_\ell']$. Since $P_\ell$ induces a linear forest, we have $|P_\ell'|\leqslant 2|I_i|+|J_i|=2|P_i|=2(r-n-1)$. Therefore, since $|P_\ell|\geqslant 3r-3n$, we have $|P_\ell\setminus (P'_\ell\cup \{e'\})|\geqslant 3r-3n-(2r-2n-2)-1=r-n+1 $. Finally, move $r-n+1$ arbitrary edges from $P_\ell\setminus (P'_\ell\cup \{e'\})$  to $P_i$. Choosing such edges guarantees that $P_i$ induces a linear forest with $2r-2n$ edges.

At the end of this process, for each $i\in [s+1,t]$, $|P_i|=2r-2n-1$ and for each $i\in [t+1,n]$, $|P_i|=2r-2n$. Also, for each $j\in [1,2n-2s]$, we have moved from $P_j$ at most $r-n$ edges whenever $t=n$ and at most $r-n+1$ edges whenever $t<n$. Therefore, at most $r-n-\epsilon+1$ edges are moved from $P_j$. Thus,  $|P_j|\geqslant 3r-3n-\epsilon -(r-n-\epsilon+1)=2r-2n-1$. Finally, for each $i\in [2n-2s+1,s]$, $|P_i|\geqslant (r-2)-e(H'\setminus H'')= r-2-t+s\geqslant 2r-2n-1$ as $n\geqslant t$ and $s\geqslant r/2$. Hence, $\mathcal{P}=(P_1,\ldots, P_n)$ satisfies the conditions of the theorem. This completes the proof.
\end{proof}

\section{Embedding Components Isomorphic to $K_2$} \label{sec:sparse}
Let $H$ be a minimal counterexample to Conjecture~\ref{conj:main} with $n$ edges. 
Suppose that $H=H'\cup (n-t)K_2$, where $H'$ is the union of all connected components of $H$ non-isomorphic to $K_2$ and let $t$ and $r$ be respectively the number of edges and vertices of $H'$. In Theorem~\ref{thm:dense}, we proved that $H'$ can be embedded into an edge decomposition of $K_r$ with certain properties. In this section, we are going to prove that such decomposition can be extended to a \hcd\ of $K_{2n+1}$ in which $H$ is rainbow and so Conjecture~\ref{conj:main} is settled completely. In fact, we prove the following theorem.

\begin{theorem} \label{thm:k2}
	Let $G$ be a graph with $t$ edges and $r$ vertices, where $G\not\cong tK_2 $ and also let $n\geqslant t$ be an integer. Suppose that $K_{r}$ admits an edge decomposition $\mathcal{P}=(P_1,\ldots, P_n)$ such that 
	\begin{itemize}
		\item for each $i\in [n]$, $P_i$ induces a linear forest in $K_{r}$,
		\item  $G$ is rainbow in $P_1,\ldots, P_t$ and
		\item for each $i\in [t]$, $|P_i|\geqslant 2r-2n-1$ and for each $i\in [t+1,n]$, $|P_i|\geqslant 2r-2n $.  
	\end{itemize}
	Then, $\mathcal{P}$ can be extended to a \hcd\ of $K_{2n+1}$ in which the graph $G\cup (n-t)K_2$ is rainbow.
\end{theorem}    

In order to prove Theorem~\ref{thm:k2}, we begin with the given decomposition $\mathcal{P}$ of $K_r$ and apply an idea by Hilton \cite{hilton} to  add vertices of $K_{2n+1}\setminus K_r$ one by one and we try to construct a rainbow matching in the remaining classes $P_i$, $t+1\leqslant i\leqslant n$. For this, we need a couple of lemmas.  First, let us review some definitions.

For an edge coloring of a multigraph $G$ with colors $1,\ldots, k$, let $C_i(v)$ be the set of edges of color $i$ incident with $v$ and let $C_i(u,v)$ be the set of edges of color $i$ with endpoints $u,v$. An edge coloring of $G$ is called a balanced edge coloring if for every vertex $v\in V(G)$,
\begin{equation*}
\max_{1\leqslant i< j\leqslant k} \big||C_i(v)|-|C_j(v)|\big|\leqslant 1,
\end{equation*}
and for every pair of vertices $u,v\in V(G)$,
\begin{equation*}
\max_{1\leqslant i< j\leqslant k} \big||C_i(u,v)|-|C_j(u,v)|\big|\leqslant 1.
\end{equation*}

We need the following result due to de Werra.
\begin{lemma} \label{lem:dewerra}{ \rm \cite{dewerra1,dewerra2,dewerra3}}
	For every bipartite multigraph $G$ and every positive integer $k$, $G$ admits a balanced $k$-edge coloring. 
\end{lemma}

 Let $G$ be a bipartite multigraph with bipartition $(X,Y)$. By a \textit{pairing} of $G$ on $X$, we mean a collection $\{\mathcal{E}_x: x\in X\}$ such that each $\mathcal{E}_x$ is a family of disjoint two-subsets of the edges incident with $x$. If $\{e_1,e_2\}\in \mathcal{E}_x$, then we say $e_1$ and $e_2$ form an $x$-pair. We have also an additional condition that if there are more than one edge between two vertices $x$ and $y$, then at most one of these parallel edges is paired with an edge $xy'$, $y'\neq y$. 
 We also need the following lemma which can be derived from the proof of Theorem~1 in \cite{hilton}. 
 For the completeness, we provide a proof of the following lemma in Appendix~\ref{sec:app}. 
 
\begin{lemma} {\rm \cite{hilton}}\label{lem:hilton}
Let $F$ be a bipartite multigraph with bipartition $(X,Y)$ such that the degree of each vertex in $Y$ is even. Also, consider a pairing of $F$ on $X$. Then, $F$ has a balanced edge-coloring with two colors such that if $e_1,e_2\in E(F)$ form an $x$-pair for some $x\in X$, then $e_1$ and $e_2$ have different colors.   
\end{lemma}

Finally, we need the following lemma.
\begin{lemma} \label{lem:reduction}
Suppose that $G$ is a bipartite multigraph with bipartition $(X,Y)$ whose edge set is partitioned into two sets $A,B$ such that for every $y\in Y$, $\deg_B(y)\geqslant \deg_A(y)$ and for every $x\in X$, $\deg_A(x), \deg_B(x)\leqslant \eta$ for some integer $\eta$. Also, suppose that for some vertex $x_0\in X$, we have either
\begin{itemize}
	\item $\deg_A(x_0)> \deg_B(x_0)$, or
	\item $\deg_A(x_0)= \deg_B(x_0)$ is odd, $\eta$ is even and for every $y\in Y$, $\deg_B(y)$ is even.
\end{itemize}   Then, there is a subset $C\subset E(G)$ such that 

\begin{itemize}
	\item $\deg_{C}(y)=\deg_A(y)$, for all $y\in Y$,
	\item $\deg_C(x_0)=\deg_A(x_0)-1$, and
	\item $\deg_A(x)\leqslant \deg_C(x)\leqslant \eta $, for all $x\in X\setminus \{x_0\}$.
\end{itemize}
\end{lemma}
\begin{proof}
Let $\widetilde{X}$ be the set of all vertices $x\in X$ such that there is an alternating path $P_x$ from $x_0$ to $x$ with edges in $A$ and $B$ alternately starting from an edge in $A$ (note that $x_0\in \widetilde{X}$).

 If for some $x\in \widetilde{X}\setminus \{x_0\}$, $\deg_A(x) < \eta$, then we set $C=(A\setminus P_x)\cup (P_x\setminus A)$ which satisfies all the conditions. Now, suppose that for all $x\in \widetilde{X}\setminus\{x_0\}$, $\deg_A(x)= \eta\geqslant \deg_B(x)$. Let $\widetilde{Y}=N_A(\widetilde{X})$ and so $e_A(\widetilde{X},\widetilde{Y})\geqslant e_B(\widetilde{X},\widetilde{Y})$. On the other hand, 
it is clear that $N_B(\widetilde{Y}) \subseteq \widetilde{X}$ and since  for every $y\in Y$, $\deg_B(y)\geqslant \deg_A(y)$, we have $e_B(\widetilde{X},\widetilde{Y}) \geqslant e_A(\widetilde{X},\widetilde{Y})$. If $\deg_A(x_0)> \deg_B(x_0)$, then we have $e_A(\widetilde{X},\widetilde{Y}) > e_B(\widetilde{X},\widetilde{Y})$ which is a contradiction. Now, suppose that  $\deg_A(x_0)=\deg_B(x_0)$ is odd, $\eta$ is even and for every $y\in Y$, $\deg_B(y)$ is even. Then, $e_B(\widetilde{X},\widetilde{Y}) = e_A(\widetilde{X},\widetilde{Y})=\eta|\widetilde{X}|-\eta+\deg_A(x_0)$ is odd and $e_B(\widetilde{X},\widetilde{Y})=\sum_{y\in \widetilde{Y}} \deg_B(y)$ is even which is again a contradiction. This contradiction completes the proof.
\end{proof}

Now, we are ready to prove Theorem~\ref{thm:k2}.

\begin{proof}[Proof of Theorem~\ref{thm:k2}.]
We start with the edge decomposition $\mathcal{P}$ of $K_r$ with the described properties and for each $s\in[0,n-t]$, we are going to prove that $\mathcal{P}$ can be extended to a decomposition  $\mathcal{Q}=(Q_1,\ldots, Q_n)$ of $K_{2s+r}$ such that $G\cup sK_2$ is rainbow in $Q_1,\ldots, Q_{t+s}$ and
\begin{equation} \label{eq:Q}
|Q_i|\geqslant 
\begin{cases}
4s+2r-2n-1 & i\in [t+s],\\
4s+2r-2n & i\in [t+s+1,n].
\end{cases}
\end{equation}
For $s=0$ we have $\mathcal{Q}=\mathcal{P}$. Suppose that it has been done until some  $s< n-t$ and we do it for $s+1$.

Set $k=2n-2s-r+1$ and suppose that the vertex set of $K_{2s+r}$ is $\{u_1,\ldots, u_{2s+r}\}$.
We define an auxiliary bipartite multigraph $\widetilde{G}$ with a bipartition $X=\{u_1,\ldots,u_{2s+r}\}$ and $Y=\{c_1,\ldots,c_n\}$ such that $c_i$ is adjacent to $u_j$ with an edge if $u_j$ is an endpoint of a path in $Q_i$ and $c_i$ is adjacent to $u_j$ with two parallel edges if $u_j$ is incident with no edge in $Q_i$.
Then, we have
\begin{align}
\deg_{\widetilde{G}}(u_j)&=k, & \forall \, j\in[2s+r], \label{eq:uj} \\
\deg_{\widetilde{G}}(c_i)&=4s+2r-2|Q_i|,  &\forall \, i\in[n]. \label{eq:ci}
\end{align}
To see \eqref{eq:uj}, for each $\ell\in\{0,1,2\}$, let $n_\ell$ be the number of sets $Q_i$ where $\deg_{Q_i}(u_j)=\ell $. Therefore, $\deg_{\widetilde{G}}(u_j)=2n_0+n_1=2(n-n_1-n_2)+n_1=2n-n_1-2n_2$. On the other hand, $u_j$ is incident with $2s+r-1$ edges in $K_{2s+r}$. So, $n_1+2n_2=2s+r-1$ and we have  $\deg_{\widetilde{G}}(u_j)= 2n-2s-r+1=k$. 

To see \eqref{eq:ci}, note that if $Q_i$ is empty, then $\deg_{\widetilde{G}}(c_i)=2(2s+r)$. Now, adding any edge to $Q_i$ reduces the degree of $c_i$ in $\widetilde{G}$ by two (because of its endpoints). Thus,  
$\deg_{\widetilde{G}}(c_i)=4s+2r-2|Q_i|$.

Therefore, by \eqref{eq:Q}, for each $i\in [s+t]$, we have $\deg_{\widetilde{G}}(c_i)\leqslant 2k$ and for each $i\in [s+t+1,n]$, $\deg_{\widetilde{G}}(c_i)\leqslant 2k-2$. 
Now, we prove the following claim.

\begin{claim} \label{claim:1}
There exist two disjoint subgraphs $G_1,G_2$ of $\widetilde{G}$ such that for $\ell=1,2$, 

\renewcommand{\theenumi}{\roman{enumi}}
\begin{enumerate}
	\item  for each $j\in[2s+r]$, $\deg_{G_\ell}(u_j)=1$,
	\item \label{con2} for each $i\in [n]$, $\deg_{G_\ell}(c_i)\leqslant 2$ and $\deg_{G_\ell}(c_{s+t+1})\leqslant 1$,
%
		\item\label{con-deg} for each $i\in[n]$, if $\deg_{\widetilde{G}}(c_i)= 2k-2x$, for some  integer $x$, then $$\deg_{G_1}(c_i)+\deg_{G_2}(c_i)\geqslant 
	\begin{cases}
	4-x & 1\leqslant i\leqslant s+t, \\
	3-x & i=s+t+1,\\
	5-x & s+t+2\leqslant i\leqslant n.
	\end{cases}
	$$
%
	
	\item \label{con7} for each $i\in [n]$, if $c_i$ is adjacent to some vertices $u_{j_{_1}}$ and $u_{j_{_2}}$ in $G_\ell$, then $u_{j_{_1}}$ and $u_{j_{_2}}$ are not the endpoints of a path in $Q_i$,
	\item \label{con8} if $c_{s+t+1}$ is adjacent to some vertex $u_{j_{_1}}$ in $G_1$ and $u_{j_{_2}}$ in $G_2$, then $u_{j_{_1}}$ and $u_{j_{_2}}$ are not the endpoints of a path in $Q_{s+t+1}$. Also, $G_1\cup G_2$ has no parallel edges. 
\end{enumerate}
\end{claim}

\begin{proof}[Proof of Claim~\ref{claim:1}]
First, let $\widehat{G}$ be obtained from $\widetilde{G}$ by duplicating each edge of $\widetilde{G}$. Also, we add a new vertex $u'$ and join it to all vertices $c_{s+t+1},\ldots, c_n$ each one with 4 parallel edges. Therefore, for all $j\in [2s+r]$, $\deg_{\widehat{G}}(u_j)=2k$ and $\deg_{\widehat{G}}(u')\leqslant 2k-2$. Also, for each $i\in [n]$, we have $\deg_{\widehat{G}}(c_i)\leqslant 4k$.

By Lemma~\ref{lem:dewerra}, there is a balanced $k$-edge coloring $\{L_1,\ldots, L_k\}$ of $\widehat{G}$ with $k=2n-2s-r+1$ colors. Fix a color class $L_q$. 
Since $\deg_{\widehat{G}}(u_j)=2k$ and $\deg_{\widehat{G}}(c_i)\leqslant 4k$, we have $\deg_{{L}_q}(u_j)=2$ and $\deg_{{L}_q}(c_i)\leqslant 4$. Also, for each $i\in [s+t]$, if $\deg_{\widetilde{G}}(c_i)=2k-2x$, then $\deg_{{L}_q}(c_i)\geqslant \lfloor (4k-4x)/k\rfloor \geqslant 4-x$ (since $k\geqslant 4$ as $r\leqslant 2t-1$). Moreover, for each $i\in [s+t+1,n]$, if $\deg_{\widetilde{G}}(c_i)=2k-2x$, then $\deg_{{L}_q}(c_i)\geqslant \lfloor (4k-4x+4)/k\rfloor \geqslant 5-x$ (as $k\geqslant 4$ and $x\geqslant 1$). 

Since $\deg_{\widehat{G}}(u')\leqslant 2k-2$, there is a color class say $L_1$, such that  $u'$  is incident with exactly one edge in $L_1$. Without loss of generality, suppose that four parallel edges between $u'$ and  $c_{s+t+1}$ are in the color classes $L_1, L_2, L_3$ and $L_4$. Now, for each $q\in [4]$, let $\widehat{L}_q$ be the edge-set obtained from $L_q$ by removing all edges incident with $u'$.

  We know that $\deg_{\widehat{L}_q}(c_{s+t+1})\leqslant 3$. If $\deg_{\widehat{L}_1}(c_{s+t+1})=3$, then we apply Lemma~\ref{lem:reduction} by setting $G=\widehat{G}\setminus u'$, $A=\widehat{L}_1$, $B=\widehat{L}_2$, $x_0=c_{s+t+1}$ and $\eta=4$ to obtain a set $C$ and define $L'_1=C$. Also, if $\deg_{\widehat{L}_1}(c_{s+t+1})\leqslant 2$, then define $L'_1=\widehat{L}_1$. It is clear that $\deg_{L'_1}(c_{s+t+1})\leqslant 2$ and $ \deg_{\widehat{L}_1}(c_{i})\leqslant \deg_{L'_1}(c_{i})\leqslant 4$, for all $i\neq s+t+1$ and $ \deg_{L'_1}(u_{j})= \deg_{\widehat{L}_1}(u_{j}) $, for all $j$. With a similar argument, using the color classes $\widehat{L}_3$ and $\widehat{L}_4$, we can find a color class $L'_3$ such that $\deg_{L'_3}(c_{s+t+1})\leqslant 2$ and $ \deg_{\widehat{L}_3}(c_{i})\leqslant \deg_{L'_3}(c_{i})\leqslant 4$, for all $i\neq s+t+1$ and $ \deg_{L'_3}(u_{j})= \deg_{\widehat{L}_3}(u_{j}) $, for all $j$. Let $L'=L'_1\cup L'_3$.  Then, we have
  \begin{itemize}
  	\item for all $j\in [2s+r]$, $\deg_{L'}(u_j)=4$, 
  	\item for each $i\in [s+t]$, if $\deg_{\widetilde{G}}(c_i)=2k-2x$, for some $x\geqslant 0$, then $8-2x\leqslant \deg_{L'}(c_i)\leqslant 8$,
  	\item if $\deg_{\widetilde{G}}(c_{s+t+1})=2k-2x$,  for some $x\geqslant 1$, then $6-2x\leqslant \deg_{L'}(c_{s+t+1})\leqslant 4$,
  	\item for each $i\in [s+t+2,n]$, if $\deg_{\widetilde{G}}(c_i)=2k-2x$, for some $x\geqslant 1$, then  $10-2x\leqslant \deg_{L'}(c_{i})\leqslant 8$ except for possibly one $c_i$, say $c_{s+t+2}$, for which $9-2x\leqslant \deg_{L'}(c_{s+t+2})\leqslant 8$. 
  \end{itemize} 

Now, we apply Lemma~\ref{lem:hilton} for the graph $F=\widehat{G}[L']$ and the following pairing: if there are two parallel edges $e_1,e_2$ joining $c_i$ to $u_j$, then $e_1$ and $e_2$ form an $i$-pair. If $c_i$ is adjacent to $u_{j}$ and $u_{j'}$, where $u_j$ and $u_{j'}$ are the endpoints of one path in $Q_i$, then $c_iu_{j}$ and $c_iu_{j'}$ form an $i$-pair. Also, if there is more than one edge incident with $c_i$ still not paired off, then such edges are paired arbitrarily (one edge may possibly be remained unpaired). Therefore, there is a balanced two-coloring for $F$, called $E_1,E_2$, such that if $c_i$ is adjacent to some vertices $u_{j_{_1}}$ and $u_{j_{_2}}$ in $E_i$, $i=1,2$, then $u_{j_{_1}}$ and $u_{j_{_2}}$ are not the endpoints of a path in $Q_i$.
Now, since $\deg_{L'}(c_{s+t+2}) \geqslant 9-2x$, one of the sets $E_1,E_2$, say $E_1$, satisfies $\deg_{E_1}(c_{s+t+2})\geqslant 5-x$. Finally, applying Lemma~\ref{lem:dewerra} to the graph $\widehat{G}[E_1]$, we find a balanced two-edge-coloring $E_{11},E_{12}$. Setting $G_1=\widetilde{G}[E_{11}]$ and $G_2=\widetilde{G}[E_{12}]$, we obtain the desired subgraphs. 
\end{proof}

\begin{figure}[ht]
	\centering
	\resizebox{.7\textwidth}{!}{%
		\begin{tikzpicture}
		\tikzstyle{every node}=[font=\normalsize]
		\begin{scope}[xshift=160]
		\node at (6.5,13) {$G_2:$};
		\draw  (9.5,12.5) -- (8.8,13.5);
		\draw  (9.5,12.5) -- (10.2,13.5);
		\draw  (7.5,13.5) -- (7.5,12.5);
		\draw [ fill={black} , line width=0.2pt ] (7.5,13.5) circle (0.15cm);
		\draw [ fill={black} , line width=0.2pt ] (7.5,12.5) circle (0.15cm);
		\node [font=\small] at (7.5,12) {$c_{s+t+1}$};
		\node [font=\small] at (7.5,14) {$u_{j_{_6}}$};
		\node [font=\small] at (8.8,14) {$u_{j_{_3}}$};
		\node [font=\small] at (10.2,14) {$u_{j_{_4}}$};
		\node [font=\small] at (9.5,12) {$c_{i}$};
		\draw [ fill={black} , line width=0.2pt ] (9.5,12.5) circle (0.15cm);
		\draw [ fill={black} , line width=0.2pt ] (8.8,13.5) circle (0.15cm);
		\draw [ fill={black} , line width=0.2pt ] (10.2,13.5) circle (0.15cm);
		\end{scope}
		\node at (6.5,13) {$G_1:$};
		\draw  (9.5,12.5) -- (8.8,13.5);
		\draw  (9.5,12.5) -- (10.2,13.5);
		\draw  (7.5,13.5) -- (7.5,12.5);
		\draw [ fill={black} , line width=0.2pt ] (7.5,13.5) circle (0.15cm);
		\draw [ fill={black} , line width=0.2pt ] (7.5,12.5) circle (0.15cm);
		\node [font=\small] at (7.5,12) {$c_{s+t+1}$};
		\node [font=\small] at (7.5,14) {$u_{j_{_5}}$};
		\node [font=\small] at (8.8,14) {$u_{j_{_1}}$};
		\node [font=\small] at (10.2,14) {$u_{j_{_2}}$};
		\node [font=\small] at (9.5,12) {$c_{i}$};
		\draw [ fill={black} , line width=0.2pt ] (9.5,12.5) circle (0.15cm);
		\draw [ fill={black} , line width=0.2pt ] (8.8,13.5) circle (0.15cm);
		\draw [ fill={black} , line width=0.2pt ] (10.2,13.5) circle (0.15cm);
		\node [font=\small] at (7.5,10.65) {$u_{j_{_5}}$};
		\draw [, dashed] (8.75,9.25) ellipse (1.75cm and 2.25cm);
		\draw  (9.75,10.25) -- (7.5,10.25);
		\draw  (9.75,9.25) -- (7.5,9.25);
		\draw  (9.75,8.25) -- (7.5,8.25);
		\draw [ fill={black} , line width=0.2pt ] (9.75,10.25) circle (0.15cm);
		\draw [ fill={black} , line width=0.2pt ] (7.5,8.25) circle (0.15cm);
		\node [font=\small] at (7.5,9.65) {$u_{j_{_6}}$};
		\draw [ color={red}] (7.5,10.25) -- (6.25,10.25);
		\draw [ color={red}] (7.5,9.25) -- (6.25,9.25);
		\draw [ color={red}] (6.25,9.25) -- (6.25,10.25);
		\draw [ fill={red} , line width=0.2pt ] (6.25,10.25) circle (0.15cm);
		\draw [ fill={red} , line width=0.2pt ] (6.25,9.25) circle (0.15cm);
		\node [font=\small] at (6,10.6) {$u_{2s+r+1}$};
		\node [font=\small] at (6,8.9) {$u_{2s+r+2}$};
		\draw [ fill={black} , line width=0.2pt ] (7.5,9.25) circle (0.15cm);
		\draw [ fill={black} , line width=0.2pt ] (7.5,10.25) circle (0.15cm);
		\node [font=\normalsize] at (8.5,6.5) {$Q_{s+t+1}$};
		\draw [ fill={black} , line width=0.2pt ] (9.75,8.25) circle (0.15cm);
		\draw [ fill={black} , line width=0.2pt ] (9.75,9.25) circle (0.15cm);
		\begin{scope}[xshift=160]
		\node [font=\small] at (7.5,10.65) {$u_{j_{_2}}$};
		\draw [, dashed] (8.75,9.25) ellipse (1.75cm and 2.25cm);
		\draw  (9.75,10.25) -- (7.5,10.25);
		\draw (9.75,9.25) -- (7.5,9.25);
		\draw (9.75,8.25) -- (7.5,8.25);
		\draw [ fill={black} , line width=0.2pt ] (9.75,10.25) circle (0.15cm);
		\draw [ fill={black} , line width=0.2pt ] (7.5,8.25) circle (0.15cm);
		\node [font=\small] at (7.5,9.65) {$u_{j_{_1}}$};
		\node [font=\small] at (9.75,9.65) {$u_{j_{_3}}$};
		\node [font=\small] at (9.75,8.65) {$u_{j_{_4}}$};
		\draw [ color={red}] (7.5,10.25) -- (6.25,9.75);
		\draw [ color={red}] (7.5,9.25) -- (6.25,9.75);
		\draw [ fill={rgb,255:red,255; green,0; blue,0} , line width=0.2pt ] (6.25,9.75) circle (0.15cm);
		\node [font=\small] at (6,10.1) {$u_{2s+r+1}$};
		\draw [ fill={black} , line width=0.2pt ] (7.5,9.25) circle (0.15cm);
		\draw [ fill={black} , line width=0.2pt ] (7.5,10.25) circle (0.15cm);
		\draw [ color={red}] (9.75,9.25) -- (11,8.75);
		\draw [ color={red}] (9.75,8.25) -- (11,8.75);
		\draw [ fill={rgb,255:red,255; green,0; blue,0} , line width=0.2pt ] (11,8.75) circle (0.15cm);
		\node [font=\small] at (11,9.1) {$u_{2s+r+2}$};
		\node [font=\normalsize] at (8.5,6.5) {$Q_i$};
		\draw [ fill={black} , line width=0.2pt ] (9.75,8.25) circle (0.15cm);
		\draw [ fill={black} , line width=0.2pt ] (9.75,9.25) circle (0.15cm);
		\end{scope}
		\end{tikzpicture}
	}%
	\caption{\label{fig:extesion}Adding edges to the sets $Q_1,\ldots,Q_n$ using auxiliary graphs $G_1,G_2$.}
\end{figure}
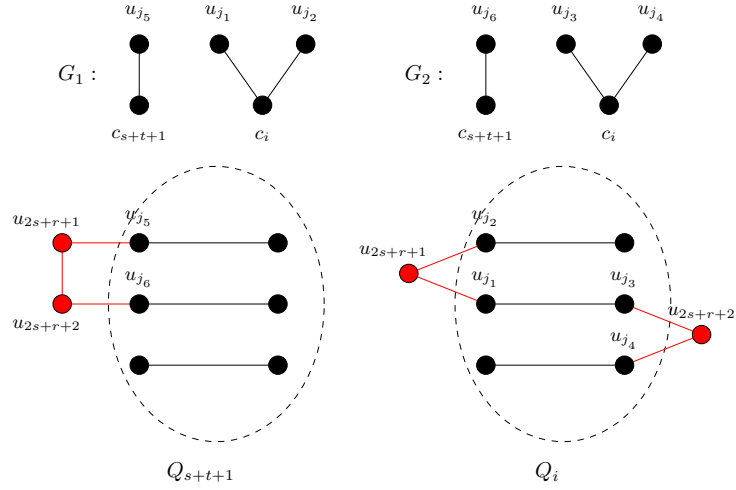

Now, we are ready to extend the decomposition $\mathcal{Q}$ of $K_{2s+r}$ to a decomposition of $K_{2s+r+2}$ by adding two new vertices $u_{2s+r+1}$ and $u_{2s+r+2}$ as follows (see Figure~\ref{fig:extesion}). For each $\ell=1,2$, if $u_j$ is adjacent to $c_i$ in $G_\ell$, then we add the edge $u_{2s+r+\ell} u_j$ to $Q_i$. Also, we add the edge $u_{2s+r+1}u_{2s+r+2}$ to $Q_{s+t+1}$.  
 So, by abuse of notation, we obtain a decomposition $\mathcal{Q}=(Q_1,\ldots,Q_n)$ for $K_{2s+r+2}$. It is clear that $G\cup (s+1)K_2$ is rainbow in $Q_1,\ldots, Q_{s+t+1}$. Moreover, because of Conditions \eqref{con2}, \eqref{con7} and \eqref{con8} of Claim~\ref{claim:1}, each $Q_i$ induces a linear forest in $K_{2s+r+2}$. Now, we check Condition \eqref{eq:Q}. First, for each $i\in [n]$, if $\deg_{\widetilde{G}}(c_i)=2k-2x$ for integer $x$, then, by \eqref{eq:ci} and Condition~\eqref{con-deg} of Claim~\ref{claim:1}, we have
 $$
 |Q_i|\geqslant
 \begin{cases}
  4s+2r-2n-1+x+4-x=4(s+1)+2r-2n-1& 1\leqslant i\leqslant s+t,\\
  4s+2r-2n-1+x+3-x+1=4(s+1)+2r-2n-1 & i=s+t+1,\\
  4s+2r-2n-1+x+5-x=4(s+1)+2r-2n & s+t+2\leqslant i\leqslant n.
 \end{cases}
 $$
 Note that when $i=s+t+1$, we have included the new edge $u_{s+t+1}u_{s+t+2}$ into $Q_{s+t+1}$ and so we have at least $4-x$ new edges therein. 
 Hence, inductively we can construct an edge decomposition $\mathcal{Q}=(Q_1,\ldots, Q_n)$ for $K_{2n-2t+r}$ such that 
 \begin{itemize}
 	\item for each $i\in [n]$, $Q_i$ induces a linear forest in $K_{2n-2t+r}$,
 	\item  $G\cup (n-t)K_2$ is rainbow in $Q_1,\ldots, Q_{n}$ and
 	\item for each $i\in [n]$, $|Q_i|\geqslant 2n-4t+2r-1$.  
 \end{itemize}
 Now, we apply Corollary~\ref{cor:hilton} to extend $\mathcal{Q}$ to a \hcd\ of $K_{2n+1}$ in which $G\cup (n-t)K_2$ is rainbow. This completes the proof.
 \end{proof}

Combining Theorems~\ref{thm:dense} and \ref{thm:k2} settles Conjecture~\ref{conj:main} as follows. Let $H$ be a minimal counterexample to Conjecture~\ref{conj:main}, where $H=H'\cup (n-t)K_2$ and $H'$ has no component isomorphic to $K_2$. Then, by Theorem~\ref{thm:dense}, we can construct an edge decomposition $\mathcal{P}$ with the desired properties and by Theorem~\ref{thm:k2}, we can extend it to a \hcd\ of $K_{2n+1}$ in which $H$ is rainbow. This contradicts the fact that $H$ is a counterexample. So, Conjecture~\ref{conj:main} is settled. 
%
%

\section{Concluding Remarks}
In this paper, we proved a conjecture by Wu asserting that for any subgraph $H$ of $K_{2n+1}$ with $n$ edges, there is a \hcd\ for $K_{2n+1}$ in which $H$ is rainbow. As it was mentioned in the introduction, this result is, in some point of view, a complete graph counterpart of Evans conjecture. In other words, the conjecture asserts that if the edges of $H$ are colored by $n$ different colors, then we can extend such coloring to a coloring of $K_{2n+1}$ such that each color class induces a Hamiltonian cycle.  So, one may ask if its generalization to an arbitrary coloring of $H$ is correct. We state this assertion as the following conjecture.

\begin{conj} \label{conj:evans}
Let $H$ be a subgraph of $K_{2n+1}$ with $n$ edges which are colored such that each color class induces a linear forest. Then, the coloring can be extended to a coloring of $K_{2n+1}$ with $n$ colors such that each color class induces a Hamiltonian cycle. 	
\end{conj}

Note that Conjecture~\ref{conj:main} is a special case of Conjecture~\ref{conj:evans} when the edges of $H$ are colored with $n$ different colors. Also, it is noteworthy that in Conjecture~\ref{conj:evans}, the number $n$ for the number of edges of $H$ is optimal, because if we  consider $H=P_2\cup P_3$ as a subgraph of $K_5$ and color the edges of $P_3$ with color $1$ and the edge of $P_2$ with color $2$, then such a coloring cannot be extended to a \hcd\ of $K_5$.

\vspace*{1cm}
\appendix
\section{Proof of Lemma~\ref{lem:hilton}} \label{sec:app}
Here we give a proof of Lemma~\ref{lem:hilton}. Let $F$ be a bipartite multigraph with bipartition $(X,Y)$ such that $e(F)=e$ and the degree of each vertex in $Y$ is even. Also, consider a pairing of $F$ on $X$ and extend it as follows. If there are two edges incident with a vertex $x\in X$ which are unpaired, then form a $x$-pair and do this until there is at most one unpaired edge on each vertex of $X$.
 
We proceed by induction on $e$. There is nothing to prove for $e=2$. So, assume that $e>2$. Now, we claim that there is an even cycle or an even path (a path with even edges) in $F$, say $Q$, 
satisfying the following properties:
\begin{enumerate}
\item if $Q$ is a path, its endpoints are in $X$, 
\item for every vertex $x\in V(Q)\cap X$, if $\deg_Q(x)=2$, then the edges of $Q$ incident with $x$ form an $x$-pair,
\item there is a balanced 2-edge coloring on $Q$ with the color classes  $ E_1 $ and $ E_2 $ such that the edges of $Q$ are alternately in $ E_1 $ and $ E_2 $.
\end{enumerate}
In order to prove the claim, select an arbitrary edge $e_0 = x_{i_0}y_{j_0}$ and place it in $ E_1 $.  If $ e_0 $ has an $ x_{i_0} $-pair, say $  e_1 = x_{i_0}y_{j_1} $, place $ e_1 $ in $ E_2 $. Since the degree of $ y_{j_1} $ is even, it has still at least one unassigned edge, let it be $ e_2 = x_{i_1}y_{j_1} $ and place $ e_2 $ in 
$ E_1 $. Continue in this way until one of the following occurs.
\begin{itemize}
\item 
An edge $ e_{2s+1} =x_{i_s}y_{j_{s+1}} $ is found (and placed in $ E_2 $), where 
$ y_{j_{s+1}}=y_{j_p} $ for some $  0 \leqslant p \leqslant s $. In this case we have an even cycle and the claim is proved.
\item  An edge $ e_{2s} =x_{i_s}y_{j_s} $ is found (and placed in $E_1$) and $ e_{2s} $ has no $x_{i_s} $-pair; then let $ x_{i_s} $ be the endpoint of this path. Now, there  is at least one unassigned edge on $ y_{j_0} $.
Let this edge be $ e_{-1} = x_{i_{-1}}y_{j_0} $ and place it in $ E_2 $. If $ e_{-1} $ has no $ x_{i_{-1}} $-pair, let $ x_{i_{-1}} $ be the other endpoint of the path $Q$. If $ e_{-1} $ has an $ x_{i_{-1}} $-pair, let it be $ e_{-2} = x_{i_{-1}}y_{j_{-1}} $ and place it in $E_1$. The vertex $ y_{j_{-1}} $ still has an 
unassigned edge on it, say $ e_{-3} = x_{i_{-2}}y_{j_{-1}} $. Place $e_{-3}$ in $ E_2 $. Continue in this way until an edge $ e_{-(2r+ 1)}=x_{i_{-r-1}}y_{j_{-r}}$ is found (and placed in $E_2$) such that the edge $ e_{-(2r+ 1)} $ has no $ x_{i_{-r-1}} $-pair. Then, let $ x_{i_{-r-1}} $ be the other endpoint of the path constructed.
\end{itemize}

Let $F'$ be the spanning subgraph of $F$ obtained by deleting the edges of $Q$. By the induction hypothesis, $F'$ has a balanced $2$-edge coloring with the color classes $E_1'$ and $E_2'$ satisfying the assertion. Clearly, the edge coloring on $F$ with the color classes $E_1''$ and $E_2''$ with $E_1''=E_1\cup E_1'$ and $E_2''=E_2\cup E_2'$ is the desired balanced $2$-edge coloring on $F$. 
\end{document}